\documentclass[a4paper,10pt,leqno]{amsart}

\title[On the universal property of Waldhausen's $K$-theory]{On the universal property of \\Waldhausen's $K$-theory}
\author{Wolfgang Steimle}
\address{Institut f\"ur Mathematik\\
	Universit\"at Augsburg\\
        D-86135 Augsburg, Germany}
\email{wolfgang.steimle@math.uni-augsburg.de}
\date{\today}

\usepackage[active]{srcltx}
\usepackage{pdfsync}
\usepackage{hyperref}
\usepackage{enumerate,amssymb, amsmath}
\usepackage{graphicx}
\usepackage[arrow,curve,matrix,tips,2cell]{xy}
  \SelectTips{eu}{10} \UseTips
  \UseAllTwocells

\DeclareMathAlphabet{\matheurm}{U}{eur}{m}{n}

\newcommand{\Wald}{\matheurm{Wald}}
\newcommand{\psset}{\matheurm{sSet}_*}
\newcommand{\fpset}{\matheurm{Set}^f_*}

\newcommand{\catinfst}{\matheurm{Cat}_\infty^{\operatorname{st}}}
\newcommand{\Eul}{\matheurm{Eul}}

\DeclareMathOperator{\Fun}{Fun}

\DeclareMathOperator{\hocolim}{hocolim}

\DeclareMathOperator{\id}{id}

\DeclareMathOperator{\iso}{iso}
\DeclareMathOperator{\map}{map}

  \newcommand{\IN}{\mathbb{N}}

  \newcommand{\cala}{\mathcal{A}}
  \newcommand{\calb}{\mathcal{B}}
  \newcommand{\calc}{\mathcal{C}}
  \newcommand{\cald}{\mathcal{D}}

  \newcommand{\cali}{\mathcal{I}}

  \newcommand{\cals}{\mathcal{S}}

  \newcommand{\bfP}{{\mathbf P}}

\theoremstyle{plain}
\newtheorem{theorem}{Theorem}[section]

\newtheorem{lemma}[theorem]{Lemma}

\theoremstyle{definition}
\newtheorem{definition}[theorem]{Definition}

\theoremstyle{remark}

\newtheorem*{remark*}{Remark}
\newtheorem*{example*}{Example}
\newtheorem*{examples*}{Examples}
\newtheorem*{notation*}{Notation}
\newtheorem*{explanation*}{Explanation}
\newtheorem*{acknowledgement}{Acknowledgement}

\newcommand{\arincl}{\ar@{^{(}->}}
\newcommand{\arinclinv}{\ar@{_{(}->}}

\newcommand{\obj}{\operatorname{ob}}

\newcommand{\Ho}{\operatorname{Ho}}

\newcommand{\naive}{{\operatorname{naive}}}

\newcommand{\sh}{\operatorname{sh}}
\newcommand{\univ}{\mathrm{univ}}

\begin{document}

\begin{abstract}
In this note we show that Waldhausen's $K$-theory functor from Waldhausen categories to spaces has a universal property: It is the target of the ``universal global Euler characteristic'', in other words, the ``additivization'' of the functor $\calc\mapsto \obj(\calc)$. We also show that other functors on the category of Waldhausen categories can be ``additivized''.
\end{abstract}

\maketitle

In the context of $\infty$-categories universal characterizations of algebraic $K$-theory have been recently given by Blumberg--Gepner--Tabuada \cite{BGT} and Barwick \cite{Barwick}. In this note we show that Waldhausen $K$-theory, in Waldhausen's original setup, has a universal property, pointing out its simplicity and generality -- in fact almost no structure on the category of Waldhausen categories is used but the bare fact that one can reasonably define an $S_\bullet$-construction. We also show that Waldhausen's $S_\bullet$-construction yields an ``additivization'' procedure for other functors.

This paper is based upon the fundamentals of Waldhausen's approach to $K$-theory \cite{Wald} and we assume that the reader is familiar with the basic definitions. Denote by $\Wald$ the category of (small) Waldhausen categories (\emph{i.e.,} categories with cofibrations and weak equivalences) and exact functors between these. 

Recall that a natural transformation $T$ between two exact functors $f, g\colon \calc\to \cald$ is  called \emph{weak equivalence} if  each $T_c\colon f(c)\to g(c)$ is a weak equivalence in $\cald$; equivalently if the adjoint of $T$ is a (necessarily exact) functor 
\[\calc\to w_1\cald\subset \Fun([1], \cald)\]
landing inside the full subcategory $w_1\cald$ on the weak equivalences $d\xrightarrow\simeq d'$. 

A sequence $f\xrightarrow{S} g\xrightarrow{T} h$ of two natural transformations between exact functors is called \emph{cofiber sequence} if each $f(c)\xrightarrow{S_c} g(c) \xrightarrow{T_c} h(c)$ is a cofiber sequence in $\cald$, and, in addition, for each cofibration $c\rightarrowtail c'$ in $\calc$, the induced map $g(c)\cup_{f(c)} f(c')\to g(c')$ is a cofibration; equivalently if $S$ and $T$ define an exact functor
\[\calc\to S_2\cald\]
into Waldhausen's extension category.

\begin{definition}\label{defn:global_euler_characteristic}
A \emph{global Euler characteristic} is a pair $(A,\chi)$ where $A\colon \Wald\to\psset$ is a functor and $\chi\colon \obj(\calc)\to A(\calc)$ is a natural transformation, such that $A$ satisfies for any Waldhausen categories $\calc$ and $\cald$:
\begin{enumerate}
 \item $A$ is \emph{product-preserving}: The canonical map $A(\calc\times\cald)\to A(\calc)\times A(\cald)$ is a weak equivalence.
 \item $A$ is \emph{homotopical}: If two exact functors $f,g\colon \calc\to \cald$ are related by a weak equivalence, then $A(f), A(g)$ are weakly homotopic. Equivalently, the canonical functor
 \[s_0\colon \calc\to w_1 \calc, \quad c \mapsto (c \xrightarrow[\simeq]{\id} c)\]
 induces a weak equivalence $A(\calc)\to  A(w_1\calc)$.
 \item $F$ is \emph{pre-additive}: If $f\rightarrowtail g\twoheadrightarrow h\colon \calc\to \cald$ is a cofibration sequence of exact functors, then $A(g)\simeq A(f\vee h)$. Equivalently, the canonical functor
 \[(d_2, d_0)\colon S_2\calc\to \calc\times\calc,\quad (c\rightarrowtail d\twoheadrightarrow e)\mapsto (c,e)\]
 induces a weak equivalence $A(S_2\calc)\to A(\calc\times\calc)$.
 \item $A$ is \emph{group-like}: The space $A(\calc)$, with the $H$-space structure given by
 \[A(\calc)\times A(\calc) \xleftarrow[\simeq]{(A(d_2), A(d_0))} A(S_2\calc) \xrightarrow{A(d_1)} A(\calc),\]
 is group-like.
\end{enumerate}
A functor $A\colon \Wald\to\psset$ satisfying these four properties will be called  \emph{additive}.
\end{definition}

\begin{remark*}
 \begin{enumerate}
  \item It follows from Segal \cite{Segal} (or the constructions in section \ref{sec:additivization}) that an additive functor $A$ naturally takes values in infinite loop spaces.
  \item Let $(A,\chi)$ be a global Euler characteristic. If $c$ is an object of some Waldhausen category $\calc$, then $\chi(c)\in A(\calc)$ is something like an Euler characteristic of $c$. Indeed, if there exists a weak equivalence $f\colon c\to d$ in $\calc$, then $\chi(c)\simeq \chi(d)$: the two elements are the images of the characteristic $\chi(c\to d)\in A(w_1\calc)$ under the two maps induced by
\[d_1, d_0\colon w_1\calc\to \calc, \quad (c\to d)\mapsto c,d;\]
but $d_1$ and $d_0$ are weakly equivalent so $F(d_1)\simeq F(d_0)$. A similar argument shows that if there exists a cofibration sequence $c\rightarrowtail d\twoheadrightarrow e$ in $\calc$, then $\chi(c)+\chi(e)\simeq \chi(d)$ for the $H$-space structure on $A(\calc)$. 
 \end{enumerate}
\end{remark*}

The prototypical example of a global Euler characteristic is $A=K:= \Omega \vert wS_\bullet - \vert$, Waldhausen's $K$-theory functor,\footnote{In our simplicial setting $\vert-\vert$ denotes taking diagonals, and $\Omega$ denotes the homotopy invariant loop space functor $\map(\Delta^1/\partial\Delta^1,f(-))$ where $f$ is a Kan replacement functor.} and $\chi_\univ\colon \obj(\calc)\to K(\calc)$ the canonical map. Our main result is that this global Euler characteristic is \emph{universal}, at least in a homotopical sense.

To formulate the result, we note that global Euler characteristics form a category $\Eul$ where a morphism $(F,\chi)\to (F', \chi')$ is  a natural transformation $a\colon F\to F'$ such that $a\circ \chi =\chi'$. We call $a$ a natural weak equivalence if each map $a_\calc\colon F(\calc)\to  F'(\calc)$ is a weak equivalence, and we denote by $\Ho(\Eul)$ the homotopy category formed by formally inverting natural weak equivalences.\footnote{This passage to the functor category and its homotopy category usually requires a change of universe, compare \cite[\S 32]{DHKS}.}

\begin{theorem}\label{thm:main}
$(K, \chi_\univ)$ is an initial object of $\Ho(\Eul)$. 
\end{theorem}

\begin{remark*}
If, in Definition \ref{defn:global_euler_characteristic}, we replace the category $\psset$ by the category of pointed sets, then $(K_0:=\pi_0 K, [\chi_{\univ}])$ is an initial object. Thus, in a precise sense, Waldhausen's functor $K$ is a homotopical analogue of $K_0$.
\end{remark*}

We can rephrase Theorem \ref{thm:main} as saying that $\chi_\univ\colon \obj(\calc)\to K(\calc)$ is the ``additivization'' of the functor $\calc\mapsto \obj(\calc)$. We also have:

\begin{theorem}\label{thm:second}
Any functor $F\colon \Wald\to \psset$ has an additivization. 
\end{theorem}

Finally, we denote by $[K,A]$ the morphisms from $K$ to $A$ in the homotopy category of the functor category $\Fun(\Wald,\psset)$, and let $\fpset$ be the Waldhausen category of finite pointed sets.

\begin{theorem}\label{thm:third}
If $A$ is additive, then the canonical map
\[[K, A] \to \pi_0 A(\fpset), \quad a \mapsto a_*(\chi_\univ(S^0))\]
is a bijection.
\end{theorem}

\begin{acknowledgement}
I am grateful to Georgios Raptis for his numerous comments on an earlier version of this draft. In particular \ref{strengthening} below is due to him.
\end{acknowledgement}

\section{Proof of Theorem \ref{thm:main}}\label{sec:proof}

Denote by $S_n\calc$ Waldhausen's category of $n$-fold extensions in $\calc$, and by $w_n\calc\subset \Fun([n],\calc)$ the full Waldhausen subcategory on strings of weak equivalences. Both constructions define, for varying $n$, simplicial objects in the category $\Wald$.

We call a functor $F$ \emph{reduced} if $F(*)=*$. For such an $F$ we let
\[TF(\calc):= \Omega\vert F(w_\bullet S_\bullet\calc)\vert\]
and 
\[\tau_F\colon F(\calc)\to TF(\calc)\]
the map which includes $\calc$ as the 0-skeleton of $w_\bullet \calc$ and which is adjoint to the inclusion of the 1-skeleton of $S_\bullet\calc$ (followed by the map into the Kan replacement). This generalizes Waldhausen's construction for the functor $\calc\mapsto \obj(\calc)$, in which case $T\obj=K$ and 
\[\chi_{\univ}:=\tau_{\obj} \colon \obj(\calc)\to K(\calc)\]
is the canonical map from above.

Waldhausen concludes from his additivity theorem that $\tau_K\colon K\to TK$ is a natural weak equivalence. His argument generalizes to show:

\begin{lemma}\label{lem:T_preserves_additive}
If $F$ is additive with $F(*)=*$ then $\tau_F\colon F\to TF$ is a natural weak equivalence.
\end{lemma}

\begin{proof}
We factor the map $\tau_F$ as a composite
\[F(\calc) \to \Omega\vert F(S_\bullet \calc)\vert \to \Omega\vert F(w_\bullet S_\bullet\calc)\vert.\]
Now the total degeneracy $F(\calc)\to F(w_n\calc)$ is a weak equivalence for any $n$ and any $\calc$, since $F$ is homotopical; it follows from the realization lemma that  the second map of the composite is a weak equivalence.

As for the  first map, we consider the diagram
\begin{equation}\label{eq:fund_square}
\xymatrix@C=0.5ex{
\bigl(F(\calc)=\bigr) & F(S_1\calc) \ar[rrrr] \ar[d]^{d_0} 
  &&&& \vert F(S_{1+\bullet}\calc)\vert \ar[d]^{d_0} & \bigl(\simeq *\bigr)
\\
\bigl(*=\bigr) &  F(S_0\calc) \ar[rrrr] &&&& \vert F(S_\bullet \calc)\vert
}
\end{equation}
where the horizontal maps are the inclusions of the $0$-skeleta. (The equivalence in the right upper corner stems from  that fact that for any simplicial space $X_\bullet$ we have $\vert X_{1+\bullet}\vert \simeq X_0$.) Waldhausen \cite[1.5]{Wald} has pointed out that if $F(*)=*$, then the map adjoint to the 1-skeletal inclusion is homotopic to the canonical map from the upper left corner in the diagram to the homotopy pull-back of the remaining square. Thus we need to show that \eqref{eq:fund_square} is a homotopy pull-back.

Now $F$ is additive so $F(S_{1+n}\calc)\simeq F(\calc)\times F(S_n\calc)$ (compare again \cite[p.~344]{Wald}) so the diagram is a homotopy pull-back in each simplicial degree. We wish to apply a theorem of Bousfield--Friedlander \cite[Theorem B.4]{BF} to conclude that the square remains a homotopy pull-back after realization. 

To verify the assumptions of this theorem, we note that by the group-like condition, 
\[d_0\colon \pi_0 F(S_{1+\bullet} \calc) \to \pi_0 F(S_\bullet\calc)\]
is a surjective map of simplicial groups and hence a Kan fibration. The same argument shows that for each $t\geq 1$, the map $[S^t, F(S_\bullet \calc)]\to \pi_0 F(S_\bullet \calc)$ is a Kan fibration. As group-like $H$-spaces are simple, \cite[B.3.1]{BF} applies to verify the $\pi_*$-Kan condition. Thus the assumption of the theorem of Bousfield--Friedlander are satisfied and its conclusion completes the proof of Lemma \ref{lem:T_preserves_additive}.
\end{proof}


To extend the definition of $T$ to non-reduced functors, we redefine $TF:=T(F/F(*))$ and $\tau_F\colon F\to F/F(*)\to T(F/F(*))$ the obvious composite. Clearly $T$ defines an endofunctor of the functor category $\Fun(\Wald,\psset)$  and $\tau$ a natural transformation from the identity functor to $T$. By the realization lemma $T$ preserves natural weak equivalences so $(T,\tau)$ extend to a functor and a natural transformation on the homotopy category (still denoted by the same letters).

Note that there are two, possibly different canonical natural transformations $K\to TK$, namely
\[\tau_K\colon K\to TK \quad\mathrm{and}\quad T\chi_\univ \colon T\obj = K \to TK,\]
the first of which is a natural weak equivalence by Lemma \ref{lem:T_preserves_additive}. Next we show that the second map is also a weak equivalence. This is a direct consequence of the following result, applied to $F=\obj$:

\begin{lemma}\label{lem:two_stabilizations_agree}
For any $F$, the two morphisms
\[\tau_{TF}, T\tau_F\colon TF\to T^2F\]
in the homotopy category differ by an automorphism of $T^2 F$.
\end{lemma}

\begin{proof}
We may assume that $F$ is reduced. There is a canonical weak equivalence
\[T^2F(\calc)= \Omega\vert \Omega \vert F(w_\bullet S_\bullet w_* S_*\calc)\vert_\bullet \vert_*\simeq \Omega^2 \vert F(w_\bullet S_\bullet w_* S_*\calc)\vert_{\bullet,*} \]
which ``pulls $\Omega$ out''. (This follows again from the Bousfield--Friedlander theorem, using that the simplicial space $\vert F(S_\bullet S_*\calc)\vert_\bullet$ is degree-wise connected \cite[p.~120]{BF}.) In the latter model for $T^2 F$ the asserted automorphism is given by the $\Sigma_2$-action which permutes $\bullet$ with $*$ and permutes the loop coordinates.
\end{proof}

To prove Theorem \ref{thm:main}, consider the following diagram in $\Ho(\Eul)$:
\begin{equation}\label{eq:main_diagram}
\xymatrix{
 && (K,\chi_\univ) \ar[d]^{\tau_K}_\cong \ar@{.>}[rr]_b
 && (A,\chi) \ar[d]^{\tau_A}_\cong
\\
(K, \chi_\univ) \ar[rr]^{T\chi_\univ}_\cong \ar@/_3ex/[rrrr]_{Ta}
 && (TK, \tau\circ\chi_\univ)  \ar@{.>}[rr]^{Tb}
 && (TA, \tau\circ\chi)
}
\end{equation}
The middle and the right vertical arrows are isomorphisms by Lemma \ref{lem:T_preserves_additive} and we just argued that $T\chi_\univ$ is so, too. But then both existence and uniqueness of $b$ follow from a simple diagram chase. 

\begin{remark*}
As a consequence of Theorem \ref{thm:main}, the two transformations $\tau_K$ and $T\chi_\univ$ from above actually agree (in the homotopy category). 
\end{remark*}

\section{The additivization functor}\label{sec:additivization}

In this section we prove Theorem \ref{thm:second}, by constructing a homotopy idempotent additivization functor $P$ on $\Fun(\Wald, \psset)$. Basically the strategy is to iterate the $S_\bullet$-construction to obtain a spectrum, of which we take  the infinite loop space. However, in showing that $P$ is homotopy idempotent, the permutation action of the symmetric groups on the iterated $S_\bullet$-construction comes into play. The theory of symmetric spectra is just made to deal with this issue and we will use some basic definitions and results of this theory as described for instance in \cite{Shipley}. 

For a reduced functor $F$ define a symmetric spectrum $\bfP F(\calc)$ with $n$-th space
\[\bfP F_n(\calc):= \vert F(w_\bullet S_\bullet^{(n)}\calc)\vert,\]
using the iterated $S_\bullet$-construction, with $\Sigma_n$ acting by permuting the $S_\bullet$-directions, and the $(n+1)$-many  structure maps $\bfP F_n(\calc) \wedge S^1\to \bfP F_{n+1}(\calc)$  are given by the inclusions of the respective $1$-skeleta of $(n+1)$ many simplicial directions. We also define
\[PF(\calc):=\Omega^\infty \bfP F(\calc)
\quad \mathrm{and} \quad
\pi_F(\calc) \colon F(\calc)\to PF(\calc)\]
the inclusion into the 0-skeleton of the $w_\bullet$ direction, followed by the inclusion of the $0$-th space of the spectrum into its infinite loop space.

\begin{explanation*} 
There is a notion of ``stable equivalence'' between symmetric spectra; the infinite-loop space functor $\Omega^\infty$ is abstractly defined as the right-derived functor (with respect to the class of stable equivalences) of the functor taking a spectrum to its 0-th space. The space $PF(\calc)=\Omega^\infty \bfP F(\calc)$ might however \emph{not} be equivalent to the space obtained by the naive  formula 
\[\Omega^\infty_\naive\bfP F(\calc):=\hocolim \bigl(F(\calc)\xrightarrow \tau \Omega \vert F(S_\bullet\calc)\vert  \xrightarrow{\tau} \Omega^2 \vert F(S_\bullet^{(2)}\calc)\vert  \xrightarrow \tau\dots \bigr);\]
indeed the functor $\Omega^\infty_\naive$ does in general not send stable equivalences of spectra to weak equivalences of spaces.
\end{explanation*}

\begin{remark*}
If the symmetric spectrum $\bfP F(\calc)$ is ``semistable'' then the naive formula will do (and there is no particular reason to use symmetric spectra at this place), but it is not clear whether this is always the case.
\end{remark*}

As in the last section, $P$ defines an endofunctor of the category $\Fun(\Wald,\psset)$ (applying $P$ to $F/F(*)$ if $F$ is not reduced) and passes to the homotopy category. Moreover $\pi$ is a natural transformation from the identity functor to $P$. By Lemma \ref{lem:T_preserves_additive} and the realization lemma, $\bfP F$ is an $\Omega$-spectrum whenever $F$ is additive (maybe up to replacing the spaces in each spectrum level by Kan complexes). In this case, $\bfP F$ is semistable so we conclude, again from Lemma \ref{lem:T_preserves_additive} that $\pi_F$ is a natural weak equivalence.

The following result generalizes \cite[1.3.5]{Wald}.

\begin{lemma}
For any $F$, the functor $PF$ is additive.
\end{lemma}

\begin{proof}
First, any functor of the form $\calc\mapsto \vert F(w_\bullet\calc)\vert$ is homotopical. Indeed, a natural weak equivalence between $f,g\colon \calc\to \cald$ induces a simplicial homotopy
\[w_\bullet \calc\times \Delta^1_\bullet \to w_\bullet\cald\]
inducing a simplicial homotopy
\[\colon F(w_\bullet \calc) \times \Delta^1_\bullet \to F(w_\bullet \cald),\]
which in turn realizes to a homotopy between $F(f)$ and $F(g)$. 

Moreover, from $F(*)\simeq *$ it follows $\bfP F(*)\simeq *$. To show additivity, denote for two Waldhausen subcategories $\cala,\calb\subset \calc$, by $E(\cala,\calc,\calb)\subset S_2\calc$ the extension category of \cite[1.1]{Wald}. It comes with boundary and degeneracy functors
\begin{align*}
d_2\colon E(\cala,\calc,\calb)\to \cala \quad\mathrm{and}\quad &d_0\colon E(\cala,\calc,\calb)\to \calb,\\
s_1\colon \cala\to E(\cala,\calc,\calb) \quad\mathrm{and}\quad & s_0\colon \calb\to E(\cala,\calc,\calb).
\end{align*}

In view of the equivalences
\[S_2\calc=E(\calc,\calc,\calc)\quad\mathrm{and}\quad 
\calc\times\cald\simeq E(\calc, \calc\times \cald, \cald),\]
additivity will follow from the assertion that
\[(\bfP F(d_2),\bfP F(d_0))\colon \bfP F(E(\cala,\calc,\calb))\to \bfP F(\cala)\times \bfP F(\calb)\]
is a stable equivalence. To do that, it is enough to show that this map induces an isomorphism on stable homotopy groups (defined using the naive formula). Indeed we will show that the map
\begin{equation}\label{eq:inverse_map}
\bfP F(s_1)_* + \bfP F(s_0)_* \colon \pi_* \bfP F(\cala)\oplus \pi_*\bfP F(\calb)\to  \pi_*\bfP F(E(\cala,\calc,\calb))
\end{equation}
is an inverse on stable homotopy groups.

Direct computation shows that \eqref{eq:inverse_map} defines a right inverse (here we use $\bfP F(*)\simeq *$ so that constant functors induce constant maps). Moreover, we note that generally that if $f\rightarrowtail g\twoheadrightarrow h$ is a cofiber sequence of exact functors $\calc\to\cald$, then the maps
\[F(g)_*\;~\mathrm{and}~\; F(f)_*+F(h)_*\colon \pi_* F(\calc)\to \pi_* F(\cald)\]
become equal after passing to $\pi_{*+1} \vert F(S_\bullet \cald)\vert$: This follows as in \cite[1.3.4]{Wald} by inspection of the 2-skeleton of $\vert F( S_\bullet\cald)\vert$. Applying this argument to the colimit system, it follows that
\[\bfP F(g)_* = \bfP F(f)_* + \bfP F(h)_*\colon \pi_* \bfP F(\calc)\to \pi_*  \bfP F(\cald).\]

We apply this observation to the cofiber sequence 
\[s_1d_2\rightarrowtail \id \twoheadrightarrow s_0 d_0\]
of endofunctors on $E(\cala,\calc,\calb)$ and conclude that \eqref{eq:inverse_map} is a left inverse.

Finally $PF$ is an infinite loop space hence group-like.
\end{proof}

We can now give the proof of Theorem \ref{thm:second}, showing more precisely that $\pi_F\colon F\to PF$ is an additivization. The argument
is the same as the one from last section, based around diagram \eqref{eq:main_diagram}, but with $(T,\tau)$ replaced by $(P,\pi)$, and $\chi_\univ\colon \obj\to K$ replaced by $\pi_F \colon F\to PF$. To make the argument work, we are left to show:

\begin{lemma}
The natural transformation $P\pi_F\colon PF\to P^2F$ is a natural weak equivalence.
\end{lemma}

\begin{proof}
Again we can assume that $F$ is reduced. By  \cite[Theorem 3.1.6]{Shipley} an explicit description of $PF(\calc)$ is as follows: 
\[PF(\calc)=\Omega^\infty_\naive(\bfP' F(\calc))\]
where $\bfP' F(\calc)$ is the symmetric spectrum obtained from $\bfP F(\calc)$ by applying the ``detection functor''. Explicitly, the $n$-th space of $\bfP' F(\calc)$ is defined as
\[\bfP' F_n(\calc) = \hocolim_{k\in \cali} \Omega^k \Sigma^n \vert F(S_\bullet^{(k)} \calc)\vert\]
where the homotopy colimit is taken over the category $\cali$ of finite ordinals and injections, which acts on both the loops and the iterated $S_\bullet$-construction. The inclusion into the homotopy colimit induces a natural transformation 
\[\eta\colon \bfP F \to \bfP' F\]
of symmetric spectra; the natural transformation $\pi_F$ is then described by the composite
\[F \xrightarrow{\iota} \bfP F_0 \xrightarrow\eta \bfP' F_0 \xrightarrow{\eta'} \hocolim_{i\in \IN} \Omega^i \bfP' F_i=PF\]
where the first map $\iota$ includes as the 0-skeleton in the $w_\bullet$-direction, and $\eta, \eta'$ are given by  the inclusions into the respective homotopy colimit. It is not hard to see that $P\iota$ is a natural equivalence, because $PF$ is homotopical. 

To conclude the proof, we will show that both maps $\eta$ and $\eta'$ are ``$\bfP$-equivalences'', \emph{i.e.}, induce a natural stable equivalence after applying $\bfP$. As homotopy colimits preserve stable equivalences \cite[4.1.5]{Shipley}, it is enough to show that, in both cases, the map into each term of the homotopy colimit is a $\bfP$-equivalence.

As for the map $\eta$, denote by $\sh^k$ denotes the $k$-fold shift functor of symmetric spectra. As in the proof of Lemma \ref{lem:two_stabilizations_agree}, the Bousfield--Friedlander theorem provides a canonical stable  equivalence (indeed $\pi_*$-equivalence)
\[\bfP (\Omega^k \vert F(S_\bullet^{(k)} -)\vert) \simeq \Omega^k \sh^k \bfP F,\]
which ``pulls the loops out''. Under this equivalence the map into the $k$-th term of the homotopy colimit identifies with the canonical map $\bfP F \to \Omega^k \sh^k \bfP F$. But it is well-known that this map is a stable equivalence for any symmetric spectrum. 

As for $\eta'$, the following diagram
\[\xymatrix{
 F \ar[r]^f \ar[d]_{\eta}^{\simeq_\bfP}  & \Omega^i\Sigma^i F \ar[d]_{\eta}^{\simeq_\bfP}  \ar[rd]^{\Omega^i \eta}_{\simeq_\bfP}\\
 \bfP' F_0 \ar[r]^(.35)f & \Omega^i \Sigma^i \bfP' F_0 \ar[r] & \Omega^i \bfP'(\Sigma^i F)_0\ar@{=}[r] & \Omega^i \bfP' F_i
}\]
is commutative, where $f$ denote the Freudenthal map. The lower composite is the map under consideration. As shown in the first step, the vertical maps and the diagonal map are $\bfP$-equivalences. Therefore it is enough to show that the Freudenthal map $f\colon F\to \Omega^i \Sigma^i F$ is a $\bfP$-equivalence.

\begin{lemma}\label{lem:connectivity_of_realization}
If $X_\bullet$ is a (semi-)simplicial object in pointed simplicial sets,  with each $X_n$ $k$-connected and $X_0$ $(k+1)$-connected, then $\vert X_\bullet \vert$ is $(k+1)$-connected.
\end{lemma}

In the interesting cases $k\geq 0$ this lemma follows by induction on the skeletal filtration of the (reduced semi-simplicial) realization using the Blakers--Massey theorem. 
Applying the Lemma repeatedly, we see that $\vert F(S_\bullet^{(n)}\calc )\vert$ is $(n-1)$-connected so that
\[\vert F(S_\bullet^{(n)}\calc )\vert \to \vert \Omega^i \Sigma^i  F(S_\bullet^{(n)}\calc )\vert\]
is roughly $2n$-connected by Freudenthal's suspension theorem (and pulling loops out once again).  Hence $\bfP f$ induces an isomorphism on $\pi_*$ and therefore is a natural stable equivalence.
\end{proof}

\begin{proof}[Proof of Theorem \ref{thm:third}]
We denote $\cals:=\fpset$ and the representable functor $R_\cals:=\Wald(\cals, -)$. This functor comes with an evaluation transformation $R_\cals(\calc)\to \obj(\calc)$, sending $f$ to $f(S^0)$; a map in the reverse direction is specified by associating to $c\in \obj(\calc)$ a natural transformation that sends $S\cup\{*\}$ to an $S$-fold wedge sum of $c$. This is not quite an inverse, though, because there may be different choices of the wedge sum. However this construction shows that the \emph{categories} $R_\cals(\calc)$ and $\calc$ are equivalent as well as their subcategories of weak equivalences; hence 
\[\vert R_\cals(w_\bullet \calc)\vert \simeq \vert w\calc\vert.\]
Plugging this into the explicit formula of the additivization functor $P$, it follows that both functors $R_\cals$ and $\obj$ have the same additivization; in other words, the natural transformation
\[R_\cals(\calc) \to K(\calc), \quad f\mapsto \chi_{\univ}(f(S^0))\]
is an additivization.

We use this to define an inverse of $e$. For $x\in A(\cals)$, consider the natural transformation
\[\alpha_x\colon R_\cals(\calc)=\Wald(\cals, \calc) \to A(\calc), \quad f \mapsto f_*(x)\]
As the target is additive, the natural transformation extends through the additivization to a unique element $a_x$ of $[K, A]$. If $x$ is homotopic to $y$, then $\alpha_x$ is homotopic to $\alpha_y$ and it follows that $a_x=a_y$. Therefore we get a well-defined map 
\[f\colon \pi_0 A(\cals)\to [K,A], \quad x\mapsto a_x.\]

By construction, $e(a_x)=x$. Moreover, if $a\in [K,A]$, we consider the commutative square
\[\xymatrix{
\pi_0 K(\cals) \ar[d]^{a_*} \ar[r]^f & [K,K] \ar[d]^{a_*}\\
\pi_0 A(\cals) \ar[r]^f & [K, A]
}\]
By construction, the element $a\in [K,A]$ is the image of $\id\in [K,K]$, which is the image of $\chi(S^0)\in \pi_0 K(\cals)$. It follows that $a$ is in the image of $f$. But $a$ was chosen arbitrarily, so $f$ is surjective. 
\end{proof}

\section{Concluding Remarks}\label{sec:concluding_remarks}

\subsection{}\label{strengthening} The proof of Theorem \ref{thm:main} actually shows that the whole \emph{space} $\hom((K,\chi_{\univ}),(A,\chi))$ of morphisms (in the hammock localization of $\Eul$ or the associated $\infty$-category) is contractible. In fact, the composite map
\begin{multline*}
\hom((K,\chi_{\univ}),(A,\chi))\xrightarrow{T} \hom((TK, \tau_K\circ\chi_{\univ}), (TA, \tau_A\circ \chi) \\
\xrightarrow{Ti^*} \hom(K,\chi_{\univ}), (TA, \tau_A\circ\chi)) 
\end{multline*}
obtained from diagram \eqref{eq:main_diagram} above is an equivalence and constant at the same time. 

\subsection{} As pointed out in the introduction, the above arguments only depend on the fact that the category of Waldhausen categories allows an $S_\bullet$-construction satisfying a few formal properties. This gives the freedom to transfer the results to different settings. For instance, we can replace the category $\Wald$ by the category $\catinfst$ of (small) stable $\infty$-categories and conclude that the functor $K(\calc)=\Omega\vert \iso S_\bullet\calc\vert$ is the additivization of the functor $\calc\mapsto \iso\calc$. 


\end{document}